\documentclass[12pt]{article}
\usepackage{amsmath,amssymb,amsthm}
\newtheorem{theorem}{Theorem}[section]

\newtheorem{lemma}[theorem]{Lemma}

\numberwithin{equation}{section}
\def\A{\mathcal{A}}

\def\B{\mathcal{B}}
\def\E{\mathcal{E}}

\begin{document}
\title{Congruence classes of triangles in $\mathbb{F}_p^2$}
\author{
    Thang Pham
  \and
    Le Anh Vinh}
\date{}
\maketitle
\begin{abstract}
In this short note, we give a lower bound on the number of congruence classes of triangles in a small set of points in $\mathbb{F}_p^2$. More precisely, for $\A\subset \mathbb{F}_p^2$ with $|\A|\le p^{2/3}$, we prove that the number of congruence classes of triangles determined by points in $\A\times \A$ is at least $|\A|^{7/2}$.  This note is not intended for journal publication.
\end{abstract}

\section{Introduction}
Let $\mathbb{F}_q$ be a finite field of order $q$ with $q=p^r$. Two $k$-simplices  in $\mathbb{F}_q^2$ with vertices $(\mathbf{x}_1, \ldots, \mathbf{x}_{k+1})$, $(\mathbf{y}_1, \ldots, \mathbf{y}_{k+1})$ are in the same congruence class if there exist an orthogonal matrix $\theta$ and a vector $\mathbf{z}\in \mathbb{F}_q^d$ such that $\mathbf{z}+\theta(\mathbf{x}_i)=\mathbf{y}_i$ for all $i=1, 2, \ldots, k+1$. Let $T_{k,d}(\E)$ denote the set of congruence classes of $k$-simplices  determined by points in $\E$. In this paper, we use the notation $X \ll Y$ which means that there exists a positive absolute constant $C$ that does not depend on $X, Y$ and $q$ such that $X\le  CY$.

The first result on the size of $T_{k,d}(\E)$ in $\mathbb{F}_q^d$ was established by Hart and Iosevich \cite{hart1}. They indicated that if $|\E|\gg q^{\frac{kd}{k+1}+\frac{k}{2}}$ with $d\ge \binom{k+1}{2}$, then $|T_{k, d}(\E)|\gg q^{\binom{k+1}{2}}$. By employing techniques from spectral graph theory, the second listed author \cite{vinh3} proved that $\E$ also contains a copy of all $k$-simplices with non-zero edges under the conditions  $d\ge 2k$ and $|\E|\gg q^{(d-1)/2+k}$. It follows from the results of Hart and Iosevich \cite{hart1} and Vinh \cite{vinh3} that we only can get a non-trivial result on the number of congruence classes of triangles when $d\ge 4$.  There are several progresses on improving this problem.  The result, which we have to mention first, is a result due to  Covert et al. \cite{cover}. They proved that if $|\E|\gg \rho q^2$, then $|T_{2,2}(\E)|\gg \rho q^3$. This means that in order to get a constant factor of all congruence classes of triangles in $\mathbb{F}_p^2$, we need the condition $|\E|\gg q^2$. In 2014, Bennett, Iosevich and Pakianathan \cite{be}, using Elekes and Sharir's approach in the Erd\H{o}s distinct distance problem, improved this result, namely, they showed that the condition $|\E|\ge q^{7/4}$ is enough to get at least $cq^3$ congruence classes of triangles for some absolute positive constant $c$. Recently, by using Fourier analytic methods and elementary results from group action theory, Bennett, Hart, Iosevich, Pakianathan, and Rudnev \cite{groupaction} improved the threshold $q^{7/4}$ to $q^{8/5}$. In the case $\E$ is the Cartesian product of sets with different sizes, the authors and Hiep \cite{hai} obtained the following improvement.

\begin{theorem}[Theorem 1.2, \cite{hai}]\label{main}
Let $\E=\A_1\times \A_2$ be a subset in $\mathbb{F}_q^2$. Suppose that 
\[(\min_{1\le i\le 2}|\A_i|)^{-1}|\E|^{k+1}\gg q^{2k},\] 
then for $1\le k\le 2$, we have \[|T_{k, 2}(\E)|\gg q^{\binom{k+1}{2}}.\]
\end{theorem}

The following is a  consequence of Theorem \ref{main}.

\begin{theorem}[\cite{hai}]\label{tot}
Let $\A, \B$ be subsets in $\mathbb{F}_q$. If $|\A|\ge q^{\frac{1}{2}+\epsilon}$ and $|\B|\ge q^{1-\frac{2\epsilon}{3}}$ for some $\epsilon\ge 0$, then we have
\[|T_{2, 2}(\A\times \B)|\gg q^3.\]
\end{theorem}

We note that  in \cite{groupaction}, Bennett et al. gave a construction of $|\A|=q^{1/2-\epsilon'}$ and $|\B|=q$, and the number of congruence classes triangles determined by $\A\times \B$ is at most  $cq^{3-\epsilon"}$ for $\epsilon">0$. On the other hand, it follows from Theorem \ref{tot} that if $|\A|<q^{1/2}$ then we must have $|\B|> q$ to guarantee that $|T_{2,2}(\A\times \B)|\gg q^3$. Hence, the condition of $|\A|\cdot |\B|$ in Theorem \ref{tot} is sharp up to a factor of $q^{\epsilon/3}$. From the construction in \cite{groupaction}, Iosevich \cite{iosevich} conjectured that the right size of a general set $\E\subseteq \mathbb{F}_q^2$ for getting at least $cq^3$ congruence classes of triangles is around $q^{3/2}$.

The main purpose of this short note is to give a lower bound on the number of congruence classes of triangles in $\A\times \A$ when $\A$ is a small set of points. 
\begin{theorem}\label{thm1}

Let $\mathbb{F}_p$ be a field of odd prime  $p$. For $\A\subseteq\mathbb{F}_p$ with $|\A|\le p^{2/3}$, we have 
\[|T_{2,2}(\A\times \A)|\ge  |\A|^{7/2} .\]
\end{theorem}

\section{Two proofs of Theorem \ref{thm1}}
\paragraph{The first proof:} Let $\mathbb{F}_p$ be a prime field of order $p$. For $\mathcal{E}\subseteq \mathbb{F}_p^2$ and $t\in \mathbb{F}_p$, we define
$\nu_{\mathcal{E}}(\lambda)$ as the cardinality of the set $\left\lbrace (\mathbf{x}, \mathbf{y})\in \mathcal{E}\times \mathcal{E}\colon ||\mathbf{x}-\mathbf{y}||=\lambda\right\rbrace.$ In order to prove Theorem \ref{thm1}, we first need the following lemmas. 
\begin{lemma}[\cite{hai}]\label{1172016}
Let $\mathcal{E}$ be a subset in $\mathbb{F}_p^2$. For a fixed $\lambda\in \mathbb{F}_p$,  denote by $H_\lambda(\mathcal{E})$ the number of hinges of the form $(\mathbf{p}, \mathbf{q_1}, \mathbf{q_2})\in \mathcal{E}\times \mathcal{E}\times \mathcal{E}$ with $||\mathbf{p}-\mathbf{q_1}||=||\mathbf{p}-\mathbf{q_2}||=\lambda$. We have the following estimate
\begin{equation}\label{2982016}\sum_{\lambda\in \mathbb{F}_p}\nu_{\mathcal{E}}(\lambda)^2\ll\frac{|\mathcal{E}|}{4}\sum_{\lambda\in \mathbb{F}_p} H_{\lambda}(\mathcal{E})+|\E|^3.\end{equation}
\end{lemma}
By applying the recent Rudnev's breakthrough on the number of incidences between points and planes in $\mathbb{F}_p^3$ in \cite{rud}, Petridis \cite{mot} obtained the following.
\begin{lemma}[\textbf{Petridis}, \cite{mot}]\label{ADidaphat}
Let $\mathbb{F}_p$ be a field of order odd prime $p$. For $\A\subseteq\mathbb{F}_p$ with $|\A|\le p^{2/3}$, we have
\[\sum_{\lambda\in \mathbb{F}_p} H_{\lambda}(\mathcal{\A\times \A})\ll |\A|^{9/2}.\]
\end{lemma}
As a consequence of Theorem \ref{ADidaphat}, Petridis \cite{mot} proved the following result on the  number of distinct distances in $\A\times \A$.

\begin{theorem}[\textbf{Petridis}, \cite{mot}]\label{A di da phat}
Let $\mathbb{F}_p$ be a field of order odd prime $p$. For $\A\subseteq \mathbb{F}_p$ with $|\A|\le p^{2/3}$, we have the number of distinct distances determined by points in $\A\times \A$ is at least $\min\{p, |\A|^{3/2}\}.$
\end{theorem}
Combining Lemma \ref{ADidaphat} and Lemma \ref{1172016}, we obtain the following.
\begin{lemma}\label{Adidaphat}
Let $\mathbb{F}_p$ be a field of order odd prime $p$. For $\A\subseteq\mathbb{F}_p$ with $|\A|\le p^{2/3}$, we have
\begin{equation}\label{eq110}\sum_{\lambda\in \mathbb{F}_p}\nu_{\A\times \A}(\lambda)^2\ll |A|^{13/2}.\end{equation}
\end{lemma}

\begin{proof}[Proof of Theorem \ref{thm1}]

On one hand, by applying the Cauchy-Schwarz inequality, we have 
\[|T_{2,2}(\A\times \A)|\ge \frac{|\A|^{12}}{N},\]
where $N$ is the number of pairs of congruence triangles.  On the other hand, two triangles, which are denoted by  $\Delta(\mathbf{a}_1, \mathbf{a}_2, \mathbf{a}_3)$ and $\Delta(\mathbf{b}_1, \mathbf{b}_2, \mathbf{b}_3)$, are in the same congruence class if there exist an orthogonal matrix $M$ and a vector $\mathbf{z}\in \mathbb{F}_p^2$ such that 
\[M\mathbf{a}_i+\mathbf{z}=\mathbf{b}_i, ~1\le i \le 3.\]
This implies that 
\[N\le |\A\times \A| \sum_{\lambda\in \mathbb{F}_p}\nu_{\A\times \A}(\lambda)^2.\]
Thus, it follows from Lemma \ref{Adidaphat} that 
 \[N\le |\A\times \A| \sum_{\lambda\in \mathbb{F}_p}\nu_{\A\times \A}(\lambda)^2=|\A|^{17/2}.\]
In other words, 
\[|T_{2,2}(\A\times \A)|\ge |\A|^{7/2},\]
and the theorem follows.
\end{proof}
\paragraph{The second proof:}
The following lemma  was suggested by Frank de Zeeuw. 
\begin{lemma}\label{phatadida}
Let $\mathbb{F}_p$ be a field of order odd prime $p$. For $\A\subseteq\mathbb{F}_p$, we have 
\[|T_{2,2}(\A\times \A)|\ge \frac{1}{6} (|\A|^2-2)\cdot |\Delta_{\mathbb{F}_p}(\A\times \A)|,\]
where $\Delta_{\mathbb{F}_p}(\A\times \A)$ is the set of distinct distances determined by points in $\A\times \A$.
\end{lemma}
\begin{proof}
Suppose that $m=|\Delta_{\mathbb{F}_p}(\A\times \A)|$, then there exist $m$ segments with distinct distances. For each segment, it is easy to check that there are at least $(|\A|^2-2)/2$ congruence classes of triangles with the same base. On the other hand, when we count congruence classes of triangles with the bases from the set of $m$ segments, each class will be repeated at most $3$ times. This implies that 
\[|T_{2,2}(\A\times \A)|\ge \frac{1}{6}(|\A|^2-2)\cdot |\Delta_{\mathbb{F}_p}(\A\times \A)|,\]
which concludes the proof of the lemma. 
\end{proof}

From Theorem \ref{A di da phat} and Lemma \ref{phatadida}, we are able to  obtain a slight weaker version of Theorem \ref{thm1} as follows. 
\begin{theorem}

Let $\mathbb{F}_p$ be a field of odd prime  $p$. For $\A\subseteq\mathbb{F}_p$ with $|\A|\le p^{2/3}$, we have 
\[|T_{2,2}(\A\times \A)|\ge \frac{1}{6} |\A|^{7/2} .\]
\end{theorem}

Let $\mathbb{F}_q$ be a finite field of order $q$. For $\E\subseteq \mathbb{F}_q^2$ with $|\E|\ge \rho q^2$ with $q^{-1/2}\ll \rho<1$. Iosevich and Rudnev \cite{adida} proved that the number of distinct distances determined by points in $\E$ is at least $q-1$.  It follows from Lemma \ref{phatadida} that the number of congruence classes of triangles in $\E$ is at least $\rho(q-1)q^2$. In other words, we have recovered  the result due to Covert et al. \cite{cover}  and Vinh \cite{vinh3} for the case $d=2$.

\begin{theorem}[\textbf{Covert et al}., \cite{cover}]
Let $\mathbb{F}_q$ be a finite field of order $q$. For $\E\subseteq \mathbb{F}_q^2$ with $|\E|\ge \rho q^2$ with $q^{-1/2}\ll \rho<1$, then the number of congruence classes of triangles in $\E$ is at least $c\rho q^3$, for some positive constant $c$.
\end{theorem}

\vspace{1cc}
\hfill\\
Department of Mathematics,\\
EPF Lausanne\\
Switzerland\\
E-mail: thang.pham@epfl.ch\\
\bigskip\\
University of Education, \\
Vietnam National University\\
Viet Nam\\
E-mail: vinhla@vnu.edu.vn

\end{document}